\documentclass[a4paper,12pt]{article}
\usepackage{amsmath,amssymb,amsfonts,amsthm,graphicx,fancyhdr,mathrsfs}
\usepackage[latin1]{inputenc}
\usepackage[T1]{fontenc}
\usepackage{rotating}
\usepackage{hyperref}
\usepackage{libertine}

\newtheorem{teo}{Theorem}

\newtheorem{cor}[teo]{Corollary}
\newtheorem{defi}[teo]{Definition}
\newtheorem{ques}[teo]{{\rm Question}}

\newtheoremstyle{drem}
     {3pt}
     {3pt}
     {\rmfamily}
     {}
     {\itshape}
     {:}
     {.5em}
     {}

\theoremstyle{drem}

\newtheorem{rmk}[teo]{Remark}

\newcommand{\vstr}[0]{{\vrule width 0in height 3ex depth 0in}}

\newcommand{\eg}[0]{\emph{e.g.} }
\newcommand{\ie}[0]{\emph{i.e.} }

\newcommand{\eps}[0]{\epsilon}

\newcommand{\rr}[0]{\ensuremath{\mathbb{R}}}
\newcommand{\zz}[0]{\ensuremath{\mathbb{Z}}}

\newcommand{\supp}[1]{\textrm{\raisebox{-.01ex}{\mbox{$\underset{#1}{\sup}$}}} \:}

\DeclareSymbolFont{bbold}{U}{bbold}{m}{n}
\DeclareSymbolFontAlphabet{\mathbbold}{bbold}
\newcommand{\un}[0]{\mathbbold{1}}

\title{Amenability criterions and critical probabilities in percolation}
\author{Antoine Gournay}
\date{\today}

\begin{document}
\maketitle

A conjecture of Benjamini and Schramm \cite[Conjecture 6]{BS} asks whether the critical percolation and unique percolation are distinct in all Cayley graphs of non-amenable groups (the definitions are given below). The aim of this note is to give a quick proof of:
\begin{teo}\label{leteo}
A countable group is non-amenable if and only if, for some finite multiset $S$ and some parameter $p$ (for Bernoulli percolation), there are infinitely many infinite percolation clusters on the Cayley graph (defined by $S$).
\end{teo}
There are already known improvements to this result: A.~Thom \cite[Corollary 8]{Thom} showed it is actually true for $S$ a generating set. Theorem \ref{leteo} was first proved by Nagnibeda \& Pak \cite[Theorem 1]{PN}. The multisets they take are powers of a given generating set, see \eg \cite[Theorem 2]{PN}. 

Like \cite{PN}, the idea is to show that the conductance (or isoperimetric) constant for some multiset is larger than $\sqrt{1/2}$. This implies $p_c<p_u$ using the work of Benjamini \& Schramm \cite{BS}. The difference in the proof given here is a criterion of amenability due to F{\o}lner, numbered Theorem \ref{tfofol} below, which seems to have been largely forgotten. The real aim of this note is to unearth it. Given this criterion (which replaces Lemma 2 of \cite{PN}), the proof is just ``connecting the dots'' between known results. 

The proof for F{\o}lner's criterion (Theorem \ref{tfolact}) is given in \S{}\ref{sfol} for completeness. 
The work of A.~Thom gives an improvement of this theorem, see \S{}\ref{sthom} below for details. 
F{\o}lner's ``forgotten'' criterion can be in a weaker form extended for actions, see \S{}\ref{serrat} for details.

{\it Acknowledgments:} I would like to thank L. Bartholdi for pointing out a mistake in Theorem 4.1 of the previous version of this paper. The current Theorem \ref{tfolact} gives the result for groups (but the proof does not work for actions). There is still a statement for actions which is more general than the classical one; see Theorem \ref{twfolseq} in \S{}\ref{serrat}. A statement analogous to Theorem \ref{tfolact} for actions remains open, see Question \ref{qcorr}.

\section{Definitions}\label{sdef}

\setcounter{teo}{0}
\renewcommand{\theteo}{\thesection.\arabic{teo}}
\renewcommand{\theques}{\thesection.\arabic{teo}}
\renewcommand{\thecor}{\thesection.\arabic{teo}}

By a finite multiset $S$ in (some set) $Y$, the reader should understand a finite sequence $\{y_i\}_{i=1}^n$ in $Y$. The abuse of notation $|S| = n$ should not create confusion. 

Given a graph $G = (X,E)$, the Bernoulli bond percolation at parameter $p$ is the random subgraph $G[p]$ given by the fact that the variables $Z_e$ (taking value $1$ if $e \in G[p]$ and $0$ else) are i.i.d. Bernoulli with parameter $p$. A cluster is a connected component of $G[p]$. Fix some root vertex $x_0 \in X$. Let $\theta(p) = \mathbb{P}(x$ belongs to an infinite cluster in $G[p])$ and $\zeta(p) = \mathbb{P}($there exists exactly one infinite cluster in $G[p])$. 
\begin{defi}
Critical values for percolation are
\[
p_c(G) = \sup \{p \in [0,1] \mid \theta(p) = 0 \} \text{ and }  p_u(G) = \inf \{p \in [0,1] \mid \zeta(p) =1 \}.
\]
\end{defi}
It is known the number of connected components (\emph{a.k.a.} clusters) is either $0$,$1$ or $\infty$. This number is clearly $0$ for $p<p_c(G)$ and $1$ for $p>p_u(G)$. Hence, there are infinitely many such clusters for some $p$ if and only if $p_c(G) < p_u(G)$. For Cayley graphs, $G$ is replaced by $(\Gamma,S)$.

Let $\un$ be the function taking value $1$ everywhere. 
If $\Gamma$ acts on $X$ (on the left), the ``natural'' action of $\Gamma$ on functions on $X$ is the left-regular action, \ie $\gamma \cdot f(x) = (\lambda_\gamma f)(x) = f(\gamma^{-1}x)$.
\begin{defi}
The action of a countable group $\Gamma$ on $X$ is said to be amenable if there exists a linear functional $m: \ell^\infty(X) \to \rr$ such that $m(\un) =1$,  $m(f) \geq 0$ if $f \geq 0$, and $m(\gamma \cdot f) = m(f)$ for any $\gamma \in \Gamma$.
\end{defi}
A group is amenable if at acts amenably on itself, \ie $X = \Gamma$. Recall the conductance constant is
\[
h(X,S) = \frac{1}{|S|} \inf_{F \subset X} \frac{\sum_{s \in S} |sF \setminus F| }{|F|}
\]
The action of $\Gamma$ on functions extends to a convolution. Given $\phi \in \ell^1(\Gamma)$ and $f \in \ell^p(X)$, $\phi * f (x) = \sum_{\gamma} \phi(\gamma) f(\gamma^{-1}x)$. Let $\un_S$ be the characteristic function of $S \subset \Gamma$. The spectral radius for $S$ is the operator norm of convolution by the uniform probability distribution on $S$:
\[
\rho(X,S) = \sup_{\|f\|_{\ell^2(X)}=1} \big\| \tfrac{1}{|S|} \un_S * f \big\|_{\ell^2(X)} = \frac{1}{|S|} \Big\| \sum_{s \in S} \lambda_s \Big\|_{\ell^2 \to \ell^2}
\]

\section{The dots}\label{sdots}

\setcounter{teo}{0}

A first important result is that Cayley graphs satisfy $p_c \leq p_u$ (\eg see \cite[Theorem 3]{BS}). Also amenability of the group implies $p_c = p_u$ (by the argument of Burton \& Keane, see the remark before \cite[Conjecture 6]{BS}). Both statements generalise to almost transitive graphs. To prove Theorem \ref{leteo}, it suffices to show that some multiset $S$ has $p_c < p_u$ when the group is non-amenable.

The first ingredient relates conductance (or isoperimetry) and spectral radius.
\begin{teo}\emph{(Mohar 1988, \cite[Theorem 2.1 and Theorem 3.1]{Mohar})} \label{tmoh}
\[
\rho(\Gamma,S) \leq \sqrt{1-h(\Gamma,S)^2} \quad \text{ and } \quad h(\Gamma,S) \geq (1-\rho(\Gamma,S)) \frac{|S|}{|S|-1}
\]
\end{teo}
Two important ingredients are from \cite{BS}:
\begin{teo}\label{tbs1} \emph{(Benjamini \& Schramm 1996, \cite[Theorem 2]{BS})} 
\[
p_c(\Gamma, S) \leq \frac{1}{|S| h(\Gamma,S)+1}. 
\]
\end{teo}
\begin{teo}\label{tbs2} \emph{(Benjamini \& Schramm 1996, \cite[Theorem 4]{BS})}
\[
\text{If } \rho(\Gamma,S) p_c(\Gamma, S)|S| < 1 \quad \text{ then } \quad p_c(\Gamma, S) < p_u(\Gamma, S).
\]
\end{teo}
The criterion of F{\o}lner which is always cited is \cite[Main Theorem (a)]{Fol55}, see Theorem \ref{tfolseq} below. Right afterwards, F{\o}lner mentions another criterion \cite[Main theorem (b)]{Fol55}:
\begin{teo}[F{\o}lner's other criterion] \label{tfofol}
A group is amenable if and only if there exists $k \in ]0,1[$ such that for any finite multiset $S \subset \Gamma$ 
\[\tag{$\textsf{I}_{k,S}$}
\exists \text{ a finite } F \subset \Gamma \text{ such that } \frac{1}{|S|} \sum_{s \in S} |sF \cap F| > k |F|.
\]
\end{teo}
Anachronistically, this result is an improvement of Juschenko \& Nagnibeda \cite[Theorem 18]{JN} (in the sense that the condition is weaker). 
A improvement is possible (thanks to the results of Thom \cite[Theorem 1]{Thom}), see \S{}\ref{sthom} for details.

\section{Proof of theorem \ref{leteo}}\label{sproof}

\setcounter{teo}{0}

Putting Theorems \ref{tmoh}, \ref{tbs1} and \ref{tbs2} together shows that 
\[
\text{if } \quad \frac{\sqrt{1-h(\Gamma,S)^2}}{h(\Gamma,S)+|S|^{-1}} < 1 \quad \text{ then } \quad p_c(\Gamma, S) < p_u(\Gamma, S).
\]
So that $h(\Gamma,S) > \sqrt{\tfrac{1}{2}}$ implies $p_c(\Gamma, S) < p_u(\Gamma, S)$.

Theorem \ref{leteo} is thus implied by the following easy corollary of Theorem \ref{tfofol}.
\begin{cor}\label{thiks-c}
A group $\Gamma$ is non-amenable if and only if $\forall \eps \in ]0,1[$ there is a finite multiset $S$ such that $h(\Gamma,S) > 1-\eps$. 
\end{cor}
\begin{proof}
Indeed, by Theorem \ref{tfofol}, if $\Gamma$ is non-amenable, then for all $k \in ]0,1[$, there exists a finite multiset $S$ such that 
\[\tag{$\not \textsf{I}_{k,S}$}  \label{niks}
\forall \text{ finite } F \subset \Gamma, \frac{1}{|S|} \sum_{s \in S} |sF \cap F| \leq k |F|.
\]
This is equivalent to 
\[
\forall \text{ finite } F \subset \Gamma, \frac{1}{|S|} \frac{\sum_{s \in S} |sF \setminus F| }{|F|} \geq 1- k.
\]
In turn, this is equivalent to $h(\Gamma,S) \geq 1-k$.
\end{proof}
To prove Theorem \ref{leteo}, just pick $k < 1 - \sqrt{1/2}$.

\section{Proof of F{\o}lner's criterion}\label{sfol}

\setcounter{teo}{0}

Here is F{\o}lner's original theorem.
\begin{teo}\emph{(F{\o}lner 1955, \cite[Main Theorem (b)]{Fol55})}\label{tfolact}
A group $\Gamma$ is a amenable if and only if there exists $k_0 \in ]0,1[$ such that for any finite multiset $S \subset \Gamma$ there is a finite set $F \subset \Gamma$ such that
\[
\frac{1}{|S|} \sum_{s \in S} |sF \cap F| > k_0 |F|.
\]
\end{teo}
Compared to what has become known as F{\o}lner sequences, this theorem is surprising in two ways. 
First, $k_0$ need not be arbitrarily close to $1$.
Second, it suffices that a fixed proportion of the elements of $S$ that displace $F$ (by some fixed proportion) in order to apply it.

As pointed out by L. Bartholdi, the author incorrectly stated this theorem for actions in the previous version of this paper. Although there is a weaker form of F{\o}lner's criterion for actions, the author now doubts that Theorem \ref{tfolact} can be stated ``as is'' for actions (see section \ref{serrat} below).

The difficult part in the original proof of F{\o}lner is to show the existence of F{\o}lner sequence from the existence of an invariant mean. Given that, the proof of Theorem \ref{tfofol} (or \ref{tfolact}) is not difficult. The reader is strongly advised to read the original papers of F{\o}lner \cite{Fol54} and \cite{Fol55}. It still seems to be a source of inspiration for recent work, \eg see Cannon, Floyd \& Parry's recent article \cite{CFP}.

To show amenability, F{\o}lner uses a preliminary result (sometimes called the Dixmier condition). 
\begin{teo}\emph{(F{\o}lner 1954, \cite[Theorem 4]{Fol54}; see also Dixmier 1950, \cite[Th\'eor\`eme 1]{Dixmier})} \label{tfolem} \\
$\Gamma$ has an amenable action on $X$ if and only if for any $h_1, \ldots, h_n \in \ell^\infty(X)$ and $\gamma_1, \ldots, \gamma_n \in \Gamma$, 
\[ \tag{\textsf{D}} \label{fofol}
H(x) = \sum_{i=1}^n \big( h_i(x) - h_i(\gamma_i^{-1} x ) \big)
\quad \text{ has } \quad 
\supp{x \in X} H(x) \geq 0
\]
\end{teo}
\begin{proof}
If there exists a function $H$ which is written as in \eqref{fofol} and $\sup H(x) \leq -\eps$ for some $\eps>0$, then there can be no invariant mean: $m(H) =0$ (by linearity) but $m(H + \eps \un) = m(H) + \eps \leq 0$ (by positivity).

On the other hand, if $H$ satisfies \eqref{fofol} implies $\sup H(x) \geq 0$, one can define ``explicitly'' an invariant mean as follows. Consider 
\[
\bar{m}(f) = \inf_{H \text{ as in \eqref{fofol}}} \sup_{x \in X} \big( f(x) + H(x) \big). 
\]
It is very easy to check that $\bar{m}(\lambda f) = \lambda \bar{m}(f)$ (for $\lambda \geq 0$), $\bar{m}(\gamma \cdot f) = \bar{m}(f)$, $f \geq 0$ implies that $\bar{m}(f) \geq 0$ and $\bar{m}(\un) =1$. It is quite straightforward (exercise!) that $\bar{m}(f+g) \leq \bar{m}(f) + \bar{m}(g)$. 

To finish, define a linear functional $m$ (with $m(\un) = 1$) associated to $\bar{m}$ using Hahn-Banach's extension theorem. This $m$ will be an invariant mean.
\end{proof}
The classical theorem of F{\o}lner will be taken for granted.
\begin{teo}[F{\o}lner sequences] \label{tfolseq}
The action of a countable group $\Gamma$ on $X$ is amenable if and only if for all $\eps \in ]0,1[$ and for any finite set $S \subset \Gamma$, there exists a finite set $F \subset X$ such that, 
\[
\text{for any } \gamma \in S, \quad \frac{ |\gamma F \cap F|}{|F|} > 1-\eps
\]
\end{teo}
One creates [what is now known as] a F{\o}lner sequence by taking a sequence of reals $\eps_n \to 0$ and a sequence of sets $S_n$ increasing to $\Gamma$. 
\begin{proof}[Proof of theorem \ref{tfolact}:] \newcommand{\supd}[1]{\displaystyle\sup_{#1}}
F{\o}lner's Theorem \ref{tfolseq} clearly implies this condition. Introduce, for a finite multiset $S$,
\[ \tag{$\not \hspace*{-.5ex}\textsf{D}_{\eps,S}$} \label{notDepsS}
\exists H \in \ell^\infty(X) \text { with } 
H = \displaystyle\sum_{s \in S} \big(  h_s - \delta_{s}*h_s \big), 
\quad
\displaystyle\vstr  2 |S| \max_{s} \|h_s\|_{\ell^\infty} = 1,
\text{~~~and~~~}
\displaystyle\vstr \sup_x H(x) \leq -\eps.
\]
Recall, for a finite multiset $S$ and $k \in ]0,1[$, the condition
\[ \tag{$\textsf{I}_{k,S}$} \label{iks}
\exists F \subset \Gamma \text{ finite such that } \sum_{\gamma \in S} \frac{|\gamma F  \cap F|}{|S| \, |F|} > k.
\]
The proof goes by contradiction: assume \eqref{notDepsS} holds (for some $\eps>0$ and some $S$) even though there is a $k$ so that \eqref{iks} is true for all $S$. Take $H_0$ as in \eqref{notDepsS}. Let $F$ associated to $S$ as in \eqref{iks}. Define $H_1$ as follows: 
\[
H_1 = \tfrac{1}{|F|} \sum_{s \in S} \un_F * h_s   - \delta_s * \un_F * h_s,
\]
where $\un_F$ is the characteristic function of $F$. This sum is apparently a sum on $|S|\, |F|$ terms of functions with 
\[
2|F|\,|S|\, \max_{s} \|h_s\|_{\ell^\infty} \leq 1 \text{ and } \sup_{x} H_1(x) \leq -\eps.
\]
However, by the choice of $F$, many terms in the sum cancel out: there are at most $2(1-k)|S|\,|F|$ left after cancellations. This means that $\|H_1\|_{\ell^\infty} \leq (1-k)$ and implies that $-(1-k) \leq -\eps$, \ie $\eps \leq 1-k$.

But one can iterate this process: $H_1$ has the same form as $H_0$ (with a different $S$), so one may define a $H_2$ from $H_1$ in a similar manner. This implies $\eps \leq (1-k)^2$. Constructing a sequence $H_m$ in a similar fashion implies $\eps \leq (1-k)^m$. If $k \in ]0,1]$, this contradicts \eqref{notDepsS} (for any $\eps>0$ and any multiset $S$). By Lemma \ref{tfolem}, there is an invariant mean.
\end{proof}

\section{Further comments}

\setcounter{teo}{0}
\renewcommand{\thesubsection}{\thesection.\Alph{subsection}}

It is straightforward to extend the above result to the following case. Let $\Gamma$ be a countable group with a non-amenable action on a countable set $X$. If the Schreier graph of $\Gamma$ acting on $X$ is almost transitive (in the sense of Benjamini \& Schramm \cite[\S{}2]{BS}, for some generating set of $G$) then there is a finite multiset $S$ so that the Schreier graph for $S$ has $p_c < p_u$. For other consequences, see also Thom's paper \cite[\S{}3.2]{Thom}.

\subsection{Thom's improvement}\label{sthom}

It is not completely clear from the proof that one may replace multisets by sets in F{\o}lner's Theorem \ref{tfofol}. One could try to use the techniques A.~Thom used to prove \cite[Theorem 1]{Thom}. In any case, one can also deduce this directly from his theorem (as was mentioned by Juschenko \& Nagnibeda in \cite[Remark 20]{JN}).
\begin{teo}\emph{(Thom \cite[Theorem 1]{Thom})}
If $\Gamma$ is a non-amenable group, then for all $\eps>0$ there exists a generating finite set $S$ such that $\rho(\Gamma,S) < \eps$.
\end{teo}
Thus, using Theorem \ref{tmoh}, if $\Gamma$ is non-amenable then $\forall \eps>0$, there exists a finite generating set $S$ with
\[
h(\Gamma,S) \geq (1 - \eps) \frac{|S|}{|S|-1} \geq 1-\eps.
\]
\begin{cor}\label{tfolset}
If there exists $k \in ]0,1[$ such that, for all generating sets $S$, 
\[\tag{$\textsf{I}_{k,S}$}
\exists F \subset \Gamma \text{ finite such that } \sum_{\gamma \in S} \frac{|\gamma F  \cap F|}{|S| \, |F|} > k.
\]
then $\Gamma$ is amenable.
\end{cor}
\begin{proof}
Indeed, by corollary \ref{thiks-c}, $h(\Gamma,S) \geq 1-k \iff $ \eqref{niks}. One gets
\[
\Gamma \text{ non-amenable } \implies \forall k \in ]0,1[, \exists \text{ a finite generating } S \subset \Gamma \text{ such that \eqref{niks} holds.}
\]
This means 
\[
\exists k \in ]0,1[ \text{ such that } \forall \text{ finite generating } S \subset \Gamma, \text{ \eqref{iks} holds} \implies \Gamma \text{ amenable}. \qedhere
\]
\end{proof}
For many other interesting open problems whose solution is expected to come soon, see Sapir's problem list \cite{Sapir}.

\subsection{Historical notes}

It is worth noting that Banach proved the extension theorem in order to show the existence of invariant means on Abelian groups. The definition he used was $\bar{m}(f) = \inf_{\mu} \sup_x (\mu * f) (x)$ where the infimum runs over finitely supported probability measures on $\Gamma$. Though the proof of F{\o}lner's criterion given above mimics the original one, it can also be adapted to show that Banach's definition also gives a semi-norm $\bar{m}$.

F{\o}lner's theorem \ref{tfofol} has been only cited once (to the author's knowledge): Kesten \cite{Kes2} uses it to show the equivalence between spectral radius $1$ and amenability. It seems likely that the subsequent simplifications of these proofs and/or lack of interest in combinatorial criterion left F{\o}lner's ``other'' criterion in oblivion. Variations of F{\o}lner's ``usual'' condition seem to have been relatively infrequent, \eg Emerson \& Greenleaf \cite{EG74} give a more restrictive assumption (rather than a weaker one). The ``usual'' F{\o}lner condition for actions of locally compact amenable groups can be found in Greenleaf \cite{Greenleaf69}.

Note that corollary \ref{thiks-c} shows that $h(\Gamma,S) \geq 1-k \iff $\eqref{niks}. The proof of Theorem \ref{tfolact} shows:
\[
\text{\eqref{notDepsS}} \implies \text{\eqref{niks} with } k = 1-\eps \iff  h(\Gamma,S) \geq \eps.
\]
To show the existence of F{\o}lner sequences, F{\o}lner \cite[\S{}5]{Fol55} actually shows:
\[
h(\Gamma,S) \geq 1-k >0 \iff \text{\eqref{niks}} \implies \text{\eqref{notDepsS} with } \eps = \frac{1-k}{2|S|}.
\]
As mentioned in the introduction, Theorem \ref{leteo} is due to \cite{PN}. They show that (in a non-amenable group) some multiset, obtained by taking arbitrarily large powers of some generating set has arbitrarily small spectral radius. In \cite[Theorem 10, Corollary 11 and Corollary 13]{JN}, this was improved to sets for a large class of non-amenable groups (\eg those admitting a normal subgroup $N \lhd \Gamma$ with $\Gamma/ N$ amenable). The full generality is due to Thom in \cite{Thom}.

\subsection{Paradoxical decompositions}

The following remark is a small digression on F{\o}lner's proof that the free group is not amenable, \cite[Theorem 6]{Fol54}. It is essentially in von Neumann's original paper \cite[Hilfsatz 2 on p.90]{vN}.
\begin{rmk} 
It is fairly easy to see that condition \eqref{notDepsS} holds if $X$ admits a paradoxical decomposition. These decompositions are given by disjoint subsets $A_1, A_2, \ldots, A_s$, $B_1, B_2, \ldots, B_t$ and $C$ of $X$ as well as elements $\{\gamma_i\}_{i=1}^s$ and $\{\eta\}_{j=1}^t$ of $\Gamma$ such that 
\[
X = \Big( \sqcup_{i=1}^s A_i \Big) \sqcup \Big( \sqcup_{j=1}^t B_j \Big) \sqcup C = \sqcup_{i=1}^s \gamma_i A_i = \sqcup_{j=1}^t \eta_j B_j.
\]
One may assume $\gamma_1$ and $\eta_1$ are the identity. The Tarski number of $X$ (for the action of $\Gamma$), $\tau(X)$, is the minimal value of $s+t$ for all such decompositions. Let 
\[
(s+t -2) H(x) = \sum_{i=2}^s ( \un_{A_i} - \gamma_i^{-1} \un_{A_i}) + \sum_{j=2}^t ( \un_{B_i} - \eta_i^{-1} \un_{B_i}).
\]
Then $\sup_x H(x) \leq -(s+t-2)^{-1}$. Thus \eqref{notDepsS} holds with $\eps \geq \tfrac{1}{2}(\tau(X)-2)^{-1}$ and $S = \{\gamma_i\}_{i=2}^s \cup \{\eta_j\}_{j=2}^t$.
This gives a lower bound on conductance in term of the Tarski number: $h(X,S) \geq \tfrac{1}{2}(\tau(X)-2)^{-1}$, compare with the upper bound from Ceccherini-Silberstein, Grigorchuk \& de la Harpe \cite[54.Proposition]{CSGdlH}
\end{rmk}


\subsection{Weaker version of F{\o}lner's criterion for actions}\label{serrat}

Here is a stronger version of Theorem \ref{tfolseq} (because it requires something weaker than a F{\o}lner sequence).
\begin{teo}[``Weak'' F{\o}lner sequences] \label{twfolseq}
The action of a countable group $\Gamma$ on $X$ is amenable if and only if there is a $\eps \in ]0,1[$ so that for any finite set $S \subset \Gamma$, there exists a finite set $F \subset X$ such that, 
\[
\text{for any } \gamma \in S, \quad \frac{ |\gamma F \cap F|}{|F|} > 1-\eps
\]
\end{teo}
Here is a reformulation in term of F{\o}lner sequences:
if there is a $c \in [0,2[$ and a sequence $F_n \subset X$ so that 
\[
 \text{for any } \gamma \in \Gamma, \quad \liminf_{n \to \infty} \frac{ |\gamma F_n \cap F_n|}{|F_n|} \leq c
\]
then the action is amenable.
\begin{proof}
The $(\implies)$ part of the proof is standard (these are the traditional F{\o}lner sets), so only the $(\Leftarrow)$ will be done here.
The proof consists in showing that the $\ell^2$-representation of the action of $\Gamma$ on $X$ has a sequence of almost-fixed vectors. Let $\pi$ be this representation, \ie $\pi_g f(x) = f( g^{-1} x)$.

Let $S_n$ be a sequence of sets which increases to $\Gamma$ and let $F_n \subset$ be the corresponding sets.
Consider the $\ell^2$-normalised characteristic functions of the sequence $F_n$: $\xi_n = \tfrac{1}{\sqrt{|F_n|}} \un_{F_n} \in \ell^2(X)$.
By hypothesis, given $\gamma \in \Gamma$, then, for $n$ large enough, 
\[
\langle \pi_g \xi_n | \xi_n \rangle 
= \sum_{x \in X} \frac{\un_{F_n} (\gamma^{-1}x) \un_{F_n}(x)}{|F_n|}  
= \frac{\gamma F_n \cap F_n}{|F_n|} \geq 1-\eps > 0
\]
Consider an free ultrafilter $\omega$ on the natural numbers as well as the ultraproduct representation $\pi^\omega$ on $(\ell_2X)^\omega$.
The sequence $\xi_n$ represents a unit vector in $(\ell_2X)^\omega$. For this vector,
\[
\langle \pi^\omega_g \xi | \xi \rangle \geq 1-\eps > 0 
\]
for any $\gamma \in \Gamma$. This means that the convex hull $K$ of the orbit of $\xi$ (under $\pi^\omega$) is separated from $0$ (by $\xi$ itself).
Let $\eta$ be the element of minimal norm in this convex hull $K$. The existence and uniqueness of $\eta$ follows from convexity of $K$ and of the $\ell^2$-norm. By uniqueness, $\pi^{\omega}_g \eta = \eta$.

But $\eta$ can be represented by a sequence $\eta_n \in \ell_2X$. Since $\eta$ is a fixed point of $\pi^\omega$, $\|\pi_g^\omega \eta- \eta\|_2 =0$. This implies that for any $\eps>0$, for any finite $F \subset X$ and any finite $S \subset G$ there are infinitely many $n$ `such that  
\[
 \| \pi_g \eta_n - \eta_n \|_2^2 \leq \eps.
\]
This shows that the representation of the action contains almost fixed vectors and implies the action is amenable.
\end{proof}
Note that the difference between Theorem \ref{tfolact} and Theorem \ref{twfolseq} is that in the former, on requires that on average the ratio $>1-\eps$ while in Theorem \ref{twfolseq} it need to be $>1-\eps$ for any $\gamma \in S$.

Lastly, these two theorems raise the question:
\begin{ques}\label{qcorr}
Let $G$ be a countable acting on $X$.
Assume there is a $\eps \in ]0,1[$ so that for any finite multiset $S \subset \Gamma$, there exists a finite set $F \subset X$ such that, 
\[
\frac{1}{|S|} \sum_{\gamma \in S} \frac{ |\gamma F \cap F|}{|F|} > 1-\eps
\]
\end{ques}
The author currently believes to be negative, but could not find any example.


\end{document}